\documentclass{article}
\usepackage{url} 
\usepackage{amsmath}
\usepackage{amssymb}
\usepackage{amsthm}
\usepackage[utf8]{inputenc}
\usepackage[shortlabels]{enumitem}
\usepackage{tikz}
\usepackage{tikz-3dplot}
\usepackage{graphicx}

\newtheorem{theorem}{Theorem}[section]
\newtheorem{lemma}[theorem]{Lemma}
\newtheorem{definition}[theorem]{Definition}
\newtheorem{prop}[theorem]{Proposition}

\newtheorem{cor}[theorem]{Corollary}

\newtheorem{question}[theorem]{Question}

\allowdisplaybreaks
\title{Distance Critical Graphs}
\author{Joshua Cooper and Gabrielle Tauscheck}
\date{\today}

\begin{document}

\maketitle

\begin{abstract}
In 1971, Graham and Pollak provided a formula for the determinant of the distance matrix of any tree on $n$ vertices. Yan and Yeh reproved this by exploiting the fact that pendant vertices can be deleted from trees without changing the remaining entries of the distance matrix. Considering failures of their argument to generalize invites the question: which graphs have the property that deleting any one vertex results in a change to some pairwise distance? We refer to such worst-case graphs as ``distance critical''. This work explores the structural properties of distance critical graphs, preservation of distance-criticality by products, and the nature of extremal distance critical graphs. We end with a few open questions. 
\end{abstract}

\section{Introduction}

Graham and Pollak (\cite{GrPo71}) famously showed that the determinant of the distance matrix of a tree $T$ on $n$ vertices -- the $n \times n$ matrix whose each $(v,w) \in V(T)\times V(T)$ entry is the ordinary graph distance between $v$ and $w$ -- depends only on $n$.  In fact, they gave a formula which, strikingly, does not depend on the tree except via $n$: $-(n-1)(-2)^{n-2}$.  These results spawned several generalizations and new directions in subsequent years; see \cite{AouHan2014} for an extensive survey on the topic.

Some of this research is concerned with finding new proofs of the Graham-Pollak Theorem.  One such example that largely inspired the present work is \cite{YanYeh2006}, providing an elegant reproof that relies on the fact that deleting pendant vertices from a tree causes all remaining pairwise distances to remain unchanged.  In order to consider the limits of their methods, it is natural to ask which graphs have the property that no vertex can be deleted without altering the distance metric on the remaining vertices.  The question can also be viewed as determining whether there exists a maximal proper induced subgraph $H$ of a graph $G$ such that $H$ embeds isometrically into $G$.  This paper introduces ``distance critical'' graphs, which are characterized by a lack of such subgraphs, and studies some of their properties.  In Section \ref{sec:preliminaries}, we give the formal definition and give a few preliminary results.  Section \ref{sec:structural} contains more in-depth analysis of the structure of distance critical graphs, and Section \ref{sec:products} investigates how distance criticality interacts with standard graph products.  Then, Section \ref{sec:extremal} delves into properties of extremal distance critical graphs.  Numerous questions remain unanswered, of which we list a few interesting ones in the final Section \ref{sec:questions}.

\section{Preliminaries} \label{sec:preliminaries}

We begin with several useful definitions.

\begin{definition}
    Given a graph $G = (V(G),E(G))$, the distance $d_G(x,y)$ between two vertices $x,y \in V(G)$ is the length of the shortest path whose endvertices are $x$ and $y$, or $\infty$ if there is no such path.  A shortest path between two vertices is called a geodesic path.
\end{definition}

\begin{definition} A graph $G$ is distance critical if, for each $v \in V(G)$, there exist $x,y \in G-v$ so that $d_G(x,y) \neq d_{G-v}(x,y)$.
\end{definition} 

For an example of a distance critical graph, see Figure \ref{fig:dodecahedron}.  Below, we present a useful reformulation of distance criticality that draws on the notion of ``determining pairs''.

\begin{definition} A pair of vertices $a,b$ is a determining pair for $v$ if $a$ and $b$ are distinct and nonadjacent, and $v$ is their unique common neighbor.
\end{definition}  

\begin{prop} \label{prop:detpair}
A connected graph $G$ is distance critical if and only if for all $v \in V(G)$, $v$ admits a determining pair $\{a,b\}$ where $a,b \in V(G)$.
\end{prop}

\begin{proof}
Let $v \in V(G)$ be arbitrary in a distance critical graph $G$.  Note that $v$ must be an internal vertex of some geodesic path $P$ which is the unique geodesic path connecting its endvertices $x$ and $y$, or else its deletion would not alter any vertex pair's distance.  Suppose $a$ and $b$ are the neighbors of $v$ in $P$.  Notice that $a$ is not adjacent to $b$; otherwise, $P$ could be shortened by deleting $v$.  Assume there exists another vertex, $w \neq v$, such that $a$ is adjacent to $w$ and $b$ is adjacent to $w$. Then the path $xPawbPy$ has the same length as $P$, contradicting the fact that it was the unique geodesic between $x$ and $y$. Therefore, $v$ must be the unique common neighbor of some two nonadjacent vertices, as desired.

For the reverse direction, suppose $G$ has the property that every vertex $v$ admits a determining pair $\{a,b\}$.  Then $d_G(a,b) = 2$, but $d_{G-v}(a,b) > 2$, so $G$ is distance critical.
\end{proof}

		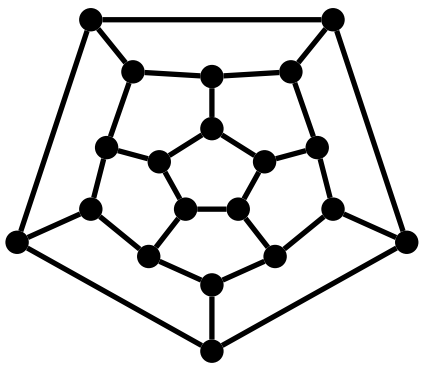
\begin{figure}[htbp]
		\centering
		\begin{tabular}{c}
			\begin{tikzpicture}
		[scale=0.7,xscale=1,yscale=0.9,rotate=180,inner sep=1mm, 
		vertex/.style={circle,thick,draw}, 
		thickedge/.style={line width=2pt}] 
			\node[vertex] (a1) at (2,3) [fill=black] {};
			\node[vertex] (a2) at (3,3) [fill=black] {};
			\node[vertex] (a3) at (3.5,2) [fill=black] {};
                \node[vertex] (a5) at (1.5,2) [fill=black] {};
                \node[vertex] (a4) at (2.5,1.3) [fill=black] {};
                \node[vertex] (p1) at (1.3,4) [fill=black] {};
			\node[vertex] (p2) at (3.7,4) [fill=black] {};
			\node[vertex] (p3) at (4.5,1.7) [fill=black] {};
                \node[vertex] (p4) at (2.5,.2) [fill=black] {};
                \node[vertex] (p5) at (0.5,1.7) [fill=black] {};
                \node[vertex] (q1) at (2.5,4.6) [fill=black] {};
			\node[vertex] (q2) at (4.8,3) [fill=black] {};
			\node[vertex] (q3) at (4,0.1) [fill=black] {};
                \node[vertex] (q4) at (1,.1) [fill=black] {};
                \node[vertex] (q5) at (0.2,3) [fill=black] {};
                \node[vertex] (o1) at (2.5,6) [fill=black] {};
			\node[vertex] (o2) at (6.2,3.7) [fill=black] {};
			\node[vertex] (o3) at (4.8,-1) [fill=black] {};
                \node[vertex] (o4) at (.2,-1) [fill=black] {};
                \node[vertex] (o5) at (-1.2,3.7) [fill=black] {};
			
			\draw[thickedge] (a1)--(a2)--(a3)--(a4)--(a5)--(a1);
                \draw[thickedge] (a1)--(p1);
                \draw[thickedge] (a2)--(p2);
                \draw[thickedge] (a3)--(p3);
                \draw[thickedge] (a4)--(p4);
                \draw[thickedge] (a5)--(p5);
                \draw[thickedge] (p1)--(q1)--(p2)--(q2)--(p3)--(q3)--(p4)--(q4)--(p5)--(q5)--(p1);
                \draw[thickedge] (q1)--(o1);
                \draw[thickedge] (q2)--(o2);
                \draw[thickedge] (q3)--(o3);
                \draw[thickedge] (q4)--(o4);
                \draw[thickedge] (q5)--(o5);
                \draw[thickedge] (o1)--(o2)--(o3)--(o4)--(o5)--(o1);
			
		\end{tikzpicture}
		\end{tabular}
	\caption{Dodecahedron}
	\label{fig:dodecahedron}
	\end{figure}

Proposition \ref{prop:detpair} immediately implies the following corollaries.

\begin{cor}\label{cor:degatleast2}
Distance critical graphs have minimum degree at least $2$.
\end{cor}

\begin{cor} \label{cor: disconnected}
A graph is distance critical if and only if all its connected components are distance critical. 
\end{cor}

Table \ref{tab:DCgraphs} depicts the number of connected distance critical graphs up to 11 vertices obtained using SageMath (\cite{sagemath}). Clearly, there are no distance critical graphs on one or two vertices since there are not enough vertices to form a determining pair. Further, there are no distance critical graphs on three or four vertices given the restrictions imposed by Proposition \ref{prop:detpair}. The smallest distance critical graph is the cycle on five vertices.

\begin{table}[htbp]\centering
	\begin{tabular}{ |c||c|c|c|c|c|c|c|c|c|c|c| } 
			\hline
$n$ & 1 & 2 & 3 & 4 & 5 & 6 & 7 & 8 & 9 & 10 & 11 \\ 
\hline
$\textrm{A}349402(n)$ & 0 & 0 & 0& 0 & 1 & 1 & 4 & 15 & 168 & 2,252 & 94,504 \\ 
			\hline
		\end{tabular}
	\caption{The number $\textrm{A}349402(n)$ of connected distance critical graphs on $n$ vertices (\cite[A349402]{oeis}).}
	
		\label{tab:DCgraphs}
	\end{table}

\section{Structural Properties} \label{sec:structural}
This section presents results on the structural properties of distance critical graphs. The {\em girth} of a graph $G$ is the length of the shortest cycle in the graph. 

\begin{lemma}\label{lem:girth}
Let $g$ represent the girth of a graph. If graph $G$  has minimum degree at least $2$ and girth $g > 4$, then it is distance critical.  
\end{lemma}

\begin{proof}
Let $G$ be a graph of minimum degree at least $2$ which is not distance critical. Since $G$ is not distance critical, Proposition \ref{prop:detpair} implies that there exists a vertex $v$ which admits no determining pair. However, $v$ must have at least $2$ neighbors, call them $a$ and $b$. If $a$ and $b$ are adjacent, then the graph contains a triangle, so that $g = 3$. Therefore, assume $a$ and $b$ are not adjacent. Since $v$ does not have a determining pair, there must be another vertex $w \neq v$ adjacent to both $a$ and $b$. Then $G$ contains the $4$-cycle $vawbv$, so that $g=4$.   
\end{proof}

We say that a connected graph $G$ is $\kappa$-connected if $\vert V(G)\vert > \kappa$ and removing any $\kappa - 1$ vertices from $G$ does not disconnect it.

\begin{lemma} \label{lem:cycle}
In a $2$-connected distance critical graph, every vertex is contained in a cycle of length at least $5$. 
\end{lemma}

\begin{proof}
Let $v \in V(G)$. Proposition \ref{prop:detpair} guarantees that $v$ has some determining pair $\{a,b\}$. Corollary \ref{cor:degatleast2} ensures that $a$ has degree at least $2$. Since $a$ and $b$ are nonadjacent, there must exist some other vertex, $x \in V(G)\setminus \{a,b,v\}$, that is adjacent to $a$. This vertex $x$ is not adjacent to $b$; otherwise, $v$ would not be the unique common neighbor of $a$ and $b$. We can use the same argument to show that there exists a vertex, $y \in V(G)\setminus\{a,b,v,x\}$, that is adjacent to $b$ but not adjacent to $a$.  Now, $G$ is $2$-connected; therefore, there is still a path, $P$, between $x$ and $y$ in $G-v$. The proof is completed via several cases. 
\begin{itemize}
    \item If $a$ and $b$ do not lie on the path $P$, then $xavbyPx$ is already a cycle in $G$ with length at least $5$. 
    \item Assume without loss of generality that $b$ lies on $P$ but $a$ does not. Since $x$ and $b$ are non-adjacent in $G$ and therefore also in $G-v$, some other vertex, call it $c$, lies between $x$ and $b$ on $P$. The graph $G$ then contains the cycle $xavbPx$ of length at least $5$. 
    \item Assume both $a$ and $b$ lie on $P$. Since $a$ and $b$ are nonadjacent, some other vertex, call it $e$, lies between $a$ and $b$ on $P$. If $e$ is the only other vertex of $aPb$, then $v$ is not the unique common neighbor of $a$ and $b$. Therefore, the path $aPb$ has length at least $3$ and $G$ contains the cycle $avbPa$ of length at least $5$. 
\end{itemize}
\end{proof}

A {\em dominating vertex} is a vertex that is adjacent to all other vertices. 

\begin{lemma} \label{nodominatingvtx} If $G$ is distance critical, then $G$ has no dominating vertex.
\end{lemma}

\begin{proof}
Assume $G$ is a distance critical graph with $v$ a dominating vertex. Let $w$ be some other vertex of the graph. By Proposition \ref{prop:detpair}, $w$ has some determining pair $\{a,b\}$.  Note that $v \not \in \{a,b\}$, since otherwise $a$ and $b$ would be adjacent. Since $v$ is a dominating vertex, $a$ and $b$ are also adjacent to $v$, contradicting the fact that $w$ is their unique common neighbor. 
\end{proof}

\begin{lemma} \label{lem:diameter}
If $G$ is a distance critical graph with $x,y \in V(G)$ such that $d_G(x,y) >3$, then $G+xy$ is distance critical as well. 
\end{lemma}

\begin{proof}
Suppose $G$ is distance critical and $G+xy$ is not. Then there exists a $v \in V(G)$ such that $v$ has a determining pair, call it $\{w,z\}$, in $G$ but not in $G+xy$. Two cases must be considered: (1) $wz \in E(G+xy)$ or (2) there exists another vertex $u \in V(G) \setminus \{v,w,z\}$ such that $uw \in E(G+xy)$ and $uz \in E(G+xy)$.

In case (1), $wz \in E(G+xy)$.  Since $\{w,z\}$ is a determining pair for $v$ in $G$, $wz \not \in E(G)$, so without loss of generality we may assume $w = x$ and $z = y$. In $G$, $\{w,z\}$ is a determining pair for $v$; therefore, $wv \in E(G + xy)$ and $vz \in E(G+ xy)$. Since $w = x$ and $z = y$, this means that $xvy$ is a $P_3$ in $G$ which implies that $d_G(x,y) \leq 2$, a contradiction. 

In case (2), both $uw \in E(G+xy)$ and $uz \in E(G+xy)$ for some vertex $u \neq v$. Without loss of generality, assume $uw \notin E(G)$ but $uw \in E(G+xy)$. Therefore, $uw = xy$. Since $uw \notin E(G)$, $uz \in E(G)$. Without loss of generality, assume $u = x$ and $w = y$. Notice that $xzvy$ is a path of length $3$ in $G$, a contradiction.   
\end{proof}
  
\begin{lemma} \label{prop:deg3}
Let $G$ be a distance critical graph and $v \in V(G)$ such that $F=G-v$ is also distance critical. If $\deg(v) \leq 3$, then $v$ is involved in some determining pair.
\end{lemma}

\begin{proof}
Corollary \ref{cor:degatleast2} ensures $\text{deg}(v) \geq 2$. Assume that $\text{deg}(v) = 2$, and label the neighbors of $v$ as $x$ and $y$. Since $G$ is distance critical, $\{x,y\}$ is the determining pair of $v$; therefore, $xy \notin E(G)$. Since $F$ is distance critical, $x$ has a determining pair that does not include $v$, call this determining pair $\{u,w\}$ where $u,w \in V(G)\setminus\{x,y,v\}$. We claimed that $\text{deg}(v) = 2$; therefore, $uv \notin E(G)$. This implies that $\{u, v\}$ is a determining pair for $x$ unless $uy \in E(G)$, which cannot occur, since otherwise $u$ would be a second common neighbor of $x$ and $y$. 

Assume $\text{deg}(v) = 3$, and label the neighbors of $v$ as $x, y$, and $z$. Since $G$ is distance critical, $v$ has a determining pair. Without loss of generality, let $\{x,y\}$ be the determining pair for $v$; therefore, $xy \notin E(G)$. Suppose that $v$ is not involved in any determining pair.  Then $x$ has a determining pair in $V(G) \setminus \{v\}$. Since $xy \notin E(G)$, at least one of these vertices, call it $u$, must be distinct from $\{x,y,z,v\}$.  Then $u$ is nonadjacent to $y$; otherwise, $\{x,y\}$ would not be a determining pair for $v$. The same argument implies that $y$ has a neighbor $w \notin \{x,y,z,v,u\}$ such that $xw \notin E(G)$. By assumption, $\text{deg}(v) = 3$; therefore, $uv \notin E(G)$ and $wv \notin E(G)$. Since $v$ is not involved in any determining pair, $u$ must be adjacent to a neighbor of $v$ other than $x$, or else $\{u,v\}$ would be a determining pair for $x$. The only possibility is if $uz \in E(G)$. The same argument implies that $wz \in E(G)$. 

From here, we prove that at least one of the pairs, $xz$ or $yz$ is not an edge. Indeed, assume $yz \in E(G)$. Then $xz \notin E(G)$ or else $\{x,y\}$ would not be a determining pair for $v$. Since $x$ has a determining pair that does not include vertex $v$, then it has at least one other  neighbor $p$ distinct from $\{x,y,z,u,v,w\}$. The fact that $\{x,y\}$ is a determining pair for $v$ implies that $py \notin E(G)$. Further, $pv \notin E(G)$ since $\text{deg}(v) = 3$, and $p$ is adjacent to $z$ since $\{v,p\}$ is not a determining pair for $x$. This implies that every neighbor of $x$ is also a neighbor of $z$. Therefore, $x$ does not have a determining pair, a contradiction. 
\end{proof}

Note a useful observation that is employed at the end of the preceding proof: in any distance critical graph, the family of neighborhoods of vertices is an antichain with respect to the subset relation.

\begin{lemma}
Let $G$ be a distance critical graph on $n$ vertices and let $S$ be the set of vertices which are involved in some determining pair. Then $\vert S\vert > \sqrt{2n}$. 
\end{lemma}

\begin{proof}
Since $G$ is distance critical, each of the $n$ vertices has a determining pair. The vertex corresponding to a determining pair is unique, so, $\binom{| S |}{2} \geq n$. Thus, $| S | > \sqrt{2n}$ as desired. 
\end{proof}

\begin{lemma}\label{lemma:dpstar}
Let $G$ be a distance critical graph and $z \in V(G)$. If $z$ has no determining pair in $G+xy$ for some $xy \not \in E(G)$, then either $\{x,y\}$ is the only determining pair for $z$ or every $z$-determining pair intersects $\{x,y\}$. 
\end{lemma}

\begin{proof}
There are two ways for the addition of the edge $xy$ to disrupt the existence of a $z$-determining pair: (1) $\{x,y\}$ is the only determining pair for $z$, or (2) $xy$ interferes with all determining pairs for $z$, i.e. for every determining pair $\{u,v\}$, $\{u,v\}$ is no longer a determining pair for $z$ in $G+xy$.  

In case (2), the edge $xy$ interferes with all determining pairs of $z$. If the addition of $xy$ destroys the $z$-determining pair $\{u,v\}$, then $u$ and $v$ must have a common neighbor $w$ other than $z$ in $G+xy$. Since this $w$ was not a neighbor of $u$ and $v$ in $G$, either $xy = uw$ or $xy = vw$, and the conclusion follows.   
\end{proof}

\section{Graph Products} \label{sec:products}

\begin{definition}
    The {\em Cartesian product} of two graphs $G$ and $H$, denoted $G \Box H$, is a graph such that $V(G\Box H) = V(G) \times V(H)$. Two vertices $(x,y)$ and $(x',y')$ are adjacent in $G\Box H$ if and only if either
    \begin{itemize}
        \item $x = x'$ and $y \sim y'$ or 
        \item $y = y'$ and $x \sim x'$. 
    \end{itemize}
\end{definition}

\begin{lemma} \label{cartesianprod}
The Cartesian product of a distance critical graph and any other graph is distance critical. 
\end{lemma}

\begin{proof}
Consider $(v,w) \in V(G \Box H)$ with $G$ distance critical and $H$ any other graph. Suppose $v \in V(G)$ has determining pair $\{a,b\}$ in $G$. Then, $(v,w)$ is adjacent to $(a,w)$ as well as to $(b,w)$. Further, $(a,w)$ and $(b,w)$ are nonadjacent because $a$ and $b$ are nonadjacent in $G$. Assume $(x,y) \neq (v,w)$ is also adjacent to $(a,w)$ and $(b,w)$. Since $(x,y)$ is adjacent to $(a,w)$, either $x=a$ or $y=w$. If $x=a$, then $w$ is adjacent to $y$. Since $(x=a,y)$ is adjacent to $(b,w)$, we have $y=w$ and $a$ is adjacent to $b$, a contradiction. Therefore $x \neq a$, and instead, $y=w$. This implies $x$ is adjacent to $a$; similarly, $x$ is adjacent to $b$. Since $\{a,b\}$ is the determining pair of vertex $v$, then $x=v$ and thus $(x,y) = (v,w)$. Therefore, every vertex of $G\Box H$ has a determining pair and the Cartesian product is distance critical as claimed.  
\end{proof}

\begin{definition}
    The {\em tensor product} of two graphs $G$ and $H$, denoted $G \times H$, is a graph such that $V(G\times H) = V(G) \times V(H)$, with two vertices $(x,y)$ and $(x',y')$ adjacent in $G\times H$ if and only if $x \sim x'$ and $y \sim y'$.
\end{definition}

\begin{lemma} \label{tensorprod}
If $G$ and $H$ are distance critical, then $G \times H$ is as well.
\end{lemma}

\begin{proof}
Assume both $G$ and $H$ are distance critical and consider $(v,w) \in V(G \times H)$. Suppose $v \in V(G)$ has determining pair $\{a,b\}$, and suppose $w \in V(H)$ has determining pair $\{c,d\}$. Since $a$ and $b$ are adjacent to $v$ and $c$ and $d$ are adjacent to $w$, this implies that $(v,w)$ is adjacent to both $(a,c)$ and $(b,d)$. Further, $(a,c)$ and $(b,d)$ are nonadjacent because $a$ and $b$ are nonadjacent in $G$. Assume there is another neighbor $(x,y)$ that is adjacent to both $(a,c)$ and $(b,d)$. This implies $x$ is adjacent to $a$ and $b$, and $y$ is adjacent to $c$ and $d$. Since $\{a,b\}$ and $\{c,d\}$ are determining pairs, this implies that $x=v$ and $y=w$. Therefore, $(v,w)$ is the unique common neighbor between vertices $(a,c)$ and $(b,d)$, and the tensor product is distance critical as claimed. 
\end{proof}

\begin{definition}
    The {\em strong product} of two graphs $G$ and $H$, denoted $G \boxtimes H$, is a graph such that $V(G\boxtimes H) = V(G) \times V(H)$, with two distinct vertices $(x,y)$ and $(x',y')$ adjacent in $G\boxtimes H$ if and only if 
    \begin{itemize}
        \item[$\circ$] $x=x'$ or $x \sim x'$ and 
        \item[$\circ$] $y = y'$ or $y \sim y'$. 
    \end{itemize}
\end{definition}

\begin{lemma} \label{strongprod}
If $G$ and $H$ are distance critical, then $G \boxtimes H$ is as well.
\end{lemma}

\begin{proof}
Assume $G$ and $H$ are distance critical graphs and consider $(v,w) \in V(G \boxtimes H)$. Suppose $v \in V(G)$ has determining pair $\{a,b\}$, and $w \in V(H)$ has determining pair $\{c,d\}$. Since $a$ and $b$ are adjacent to $v$ and $c$ and $d$ are adjacent to $w$, this implies that $(v,w)$ is adjacent to both $(a,c)$ and $(b,d)$. Further, $(a,c)$ and $(b,d)$ are nonadjacent because $a$ and $b$ are nonadjacent in $G$. Assume there is another neighbor $(x,y)$ that is adjacent to both $(a,c)$ and $(b,d)$. Since $(x,y)$ is adjacent to $(a,c)$, either $x=a$ or $x$ is adjacent to $a$. If $x=a$, then $y$ must be adjacent to $c$. Further, we are assuming that $(x=a, y)$ is adjacent to $(b,d)$, so $a=b$ or $a$ is adjacent to $b$, both which contradict $\{a,b\}$ being a determining pair.  Therefore, $x \neq a$, and we can assume that $x$ is adjacent to $a$. The same argument implies that $x$ is adjacent to $b$. Since $\{a,b\}$ is a determining pair, this implies $x=v$. The same argument applies to the second coordinate so that $y=w$ and $(v,w)$ will be the unique common neighbor between $(a,c)$ and $(b,d)$. Therefore, $G\boxtimes H$ is distance critical as claimed.  
\end{proof}

Note that we must assume $G$ and $H$ are both distance critical in the two preceding results; indeed, $C_5 \times C_4$ and $C_5 \boxtimes C_4$ are not distance critical, though $C_5$ is.

\section{Extremal Results} \label{sec:extremal}

The definition of distance critical relies on a local property (requiring that each vertex has a determining pair), so it is unclear what the maximum edge density of a distance critical graph is.  Notice that a graph being distance critical does not imply that it is (edge) maximal with respect to this property. Consider, for example, the cycle on $8$ vertices. $C_8$ is clearly a distance critical graph; however, edges can be added without disrupting its determining pairs. See, for example, Figure \ref{fig:maxC8}.  

		\begin{figure}[htbp]
		\centering
		\begin{tabular}{c}
			\begin{tikzpicture}
		[scale=1.2,xscale=1,yscale=0.9,rotate=180,inner sep=1mm, 
		vertex/.style={circle,thick,draw}, 
		thickedge/.style={line width=2pt}] 
			\node[vertex] (a1) at (2,3) [fill=black] {};
			\node[vertex] (a2) at (3,3) [fill=black] {};
			\node[vertex] (a3) at (4,2) [fill=black] {};
			\node[vertex] (a4) at (4,1) [fill=black] {};
			\node[vertex] (a5) at (3,0) [fill=black] {};
			\node[vertex] (a6) at (2,0) [fill=black] {};
			\node[vertex] (a7) at (1,1) [fill=black] {};
			\node[vertex] (a8) at (1,2) [fill=black] {};
			
			\draw[thickedge] (a1)--(a2)--(a3)--(a4)--(a5)--(a6)--(a7)--(a8)--(a1);
                \draw[thickedge] (a1)--(a5);
                \draw[thickedge] (a2)--(a6);
                \draw[thickedge] (a3)--(a7);
                \draw[thickedge] (a4)--(a8);
			
		\end{tikzpicture}
		\end{tabular}
	\caption{Maximal distance critical graph on $8$ vertices containing $C_8$ as a proper subgraph.}
	\label{fig:maxC8}
	\end{figure}
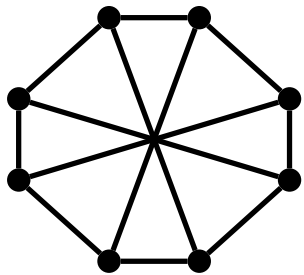

For the following three lemmas, let $n$ represent the number of vertices in the graph. A {\em regular graph} is a graph in which all vertices have the same degree.  First, we note that the fewest number of edges a distance critical graph can have is $n$.

\begin{lemma} \label{lem:minedges}
Every distance critical graph has a minimum of $n$ edges, and $C_n$ achieves this bound for $n \geq 5$. 
\end{lemma}

\begin{proof}
Every vertex of a distance critical graph has a determining pair. Therefore, the degree of every vertex is at least $2$. Thus, $|E(G)| \geq 2 |V(G)|/2 = n$. 
\end{proof}

\begin{lemma} \label{lem:maxdegree}
The maximum degree of any distance critical graph $G$ on $n \geq 6$ vertices is at most $n-4$. 
\end{lemma}

\begin{proof}
Proposition \ref{nodominatingvtx} guarantees that no vertex has degree $n-1$. 

Suppose $v \in V(G)$ had $\deg(v) = n-2$. Then $v$ is adjacent to every vertex except one; label this exception as $u$. Let $S$ be the set of vertices that are adjacent to $v$. Vertex $u$ has a determining pair $\{x,w\}$ since $G$ is distance critical. Since $u$ is not adjacent to $v$, $x, w \in S$. Every vertex in $S$, however, is adjacent to $v$, so $\{x,w\}$ is not the determining pair of $u$, and $v$ cannot have degree $n-2$. 

Assume $\deg(v) = n-3$. Let $S$ be the set of $n-3 \geq 3$ vertices adjacent to $v$, and label the remaining two vertices $u_1$ and $u_2$. Let $w \in S$. The options for a determining pair for $w$ are (a) $\{v,u_1\}$, (b) $\{v,u_2\}$, (c) $\{u_1,u_2\}$, (d) $\{x,u_1\}$, and (e)$\{x,u_2\}$  where $x$ is some other vertex in $S$. Suppose $w$ has option (a) for its determining pair. Then $w$ is adjacent to $u_1$ and $u_1$ is nonadjacent to every other vertex in $S$. Therefore, the remaining vertices of $S$ have determining pairs given by either options (b) or (e). In either case, this implies $u_2$ is adjacent to one other vertex in $S$ and nonadjacent to all other vertices of $S$. Therefore, the remaining vertices of $S$ do not have a determining pair, so option (a) and, similarly, option (b) are impossible. 

Consider $u_1$. Since $u_1$ is not adjacent to $v$, therefore, if $u_1$ and $u_2$ are nonadjacent, $u_1$ must have a determining pair with both vertices in $S$. This contradicts that $u_1$ is the unique neighbor of the vertices in the determining pair, since $v$ is adjacent to every element of $S$. Therefore, $u_1$ is adjacent to $u_2$ and some other vertex in $S$.  But then $\{u_1,u_2\}$ cannot be a determining pair for $w$, and we may rule out option (c) as well.

The vertices of $S$, therefore, only have determining pairs with either option (d) or (e). Assume $w \in S$ with determining pair (d). Then $w$ is adjacent to $u_1$ and $x$ where $x$ is another vertex of $S$ with determining pair of type (e). Therefore, $x$ is adjacent to $u_2$ and $u_1$ is adjacent to $u_2$ so $w$ is not their unique common neighbor. The same argument applies if we start with a vertex in $S$ with determining pair of type (e). Therefore, there exist vertices without a determining pair, a contradiction. 
\end{proof}

The above bound is in fact tight for $n \geq 6$, as shown by the following construction.  Suppose $n \geq 8$ is even.  Let $U = \{u_1,\ldots,u_{n/2-2}\}$ and $U' = \{u'_1,\ldots,u'_{n/2-2}\}$, let $V(G) = U \cup U' \cup \{v,w_1,w_2,w_3\}$, and define
\begin{align*}
    E(G) &= \{u_ju'_j : u_j \in U, u'_j \in U' \} \cup \{w_1 u : u \in U \} \cup \{w_2 u' : u' \in U' \} \\
    & \qquad \cup \{v w : w \in U \cup U' \} \cup \{w_1w_3,w_2w_3\}
\end{align*}
It is straightforward to check that $G$ is distance critical and $\deg(v) = n-4$.    See Figure \ref{fig:maxdegree}.  

		\begin{figure}[htbp]
		\centering
		\begin{tabular}{c}
			\begin{tikzpicture}
		[scale=1.2,xscale=1,yscale=0.9,rotate=180,inner sep=.7mm, 
		vertex/.style={circle,draw}, 
		thickedge/.style={line width=2pt}] 
			\node[vertex] (a1) at (0,0) [fill=black] {};
   \node[vertex] (a2) at (-1,1) [fill=black] {};
   \node[vertex] (a3) at (-2,1) [fill=black] {};
   \node[vertex] (a4) at (-3,1) [fill=black] {};
   \node[vertex] (a5) at (-4,1) [fill=black] {};
   \node[vertex] (a6) at (-1,-1) [fill=black] {};
   \node[vertex] (a7) at (-2,-1) [fill=black] {};
   \node[vertex] (a8) at (-3,-1) [fill=black] {};
   \node[vertex] (a9) at (-4,-1) [fill=black] {};
   \node[vertex] (w1) at (-5,-.4) [fill=black] {};
   \node[vertex] (w2) at (-5,.4) [fill=black] {};
   \node[vertex] (w3) at (-6,-0) [fill=black] {};

			\draw (a1)--(a2);
   \draw (a1)--(a3);
   \draw (a1)--(a4);
   \draw (a1)--(a5);
   \draw (a1)--(a6);
   \draw (a1)--(a7);
   \draw (a1)--(a8);
   \draw (a1)--(a9);
   \draw (w1)--(a2);
   \draw (w1)--(a3);
   \draw (w1)--(a4);
   \draw (w1)--(a5);
   \draw (w2)--(a6);
   \draw (w2)--(a7);
   \draw (w2)--(a8);
   \draw (w2)--(a9);
   \draw (w2)--(w3);
   \draw (w1)--(w3);
   \draw (a2)--(a6);
   \draw (a3)--(a7);
   \draw (a4)--(a8);
   \draw (a5)--(a9);

		\end{tikzpicture}
		\end{tabular}
	\caption{Maximum degree construction for $n=12$.}
	\label{fig:maxdegree}
	\end{figure}
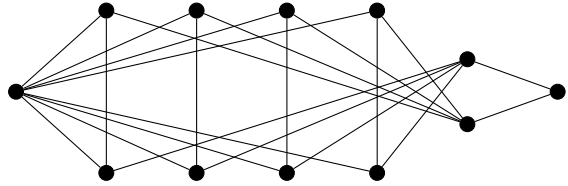
 
 If $n$ is odd, then the same construction works with an additional vertex $w_4$ and new edges $vw_4$ and $w_3w_4$.  If $n=6$, we can take $G = C_6$, and for $n=7$, consider adding to $C_6$ a single vertex adjacent to an antipodal pair.  (If $n=5$, the graph $C_5$ is the only distance critical graph, and its maximum degree is $n-3$.)

The $k^{th}$ power $G^k$ of an undirected graph $G$ is another graph on the same set of vertices, in which two vertices are adjacent when their distance in $G$ is at most $k$. 

\begin{lemma} \label{lem:regular}
If $n \geq 5$, the maximum $d$ for which there is a $d$-regular distance critical graph is $\lfloor \frac{n-1}{4}\rfloor + \lfloor \frac{n}{4}\rfloor$.
\end{lemma}

\begin{proof}
First, we prove the upper bound. Every vertex in a distance critical graph, $G$, has a determining pair. Let $\{x,y\}$ be the determining pair for some vertex $v$. The neighborhoods of $x$ and $y$ only intersect at $v$. Therefore, the remaining $n-3$ vertices are nonadjacent to at least one of $x$ or $y$. Further, since $x$ and $y$ are both adjacent to $v$, we have that $\deg(x)+\deg(y) \leq n-3+2 = n-1$. In a regular graph, all vertices have the same degree; therefore, $\deg(x) = \deg(y)$ so that $\deg(x) \leq \frac{n-1}{2}$ as desired.   Note that, if $n \not \equiv 3 \pmod{4}$, then $\lfloor (n-1)/2 \rfloor = \lfloor \frac{n-1}{4}\rfloor + \lfloor \frac{n}{4}\rfloor$.  On the other hand, if $n \equiv 3 \pmod{4}$, then both $n$ and $(n-1)/2$ are odd, so there is no $(n-1)/2$-regular graph on $n$ vertices, and we may conclude that $\deg(x) \leq \frac{n-1}{2} - 1 = \lfloor \frac{n-1}{4}\rfloor + \lfloor \frac{n}{4}\rfloor$.

Now, we construct a regular distance critical graph for each $n \geq 5$ to show that this bound is tight. Let $G$ be the cycle on $n$ vertices where vertex $i$ is adjacent to $i \pm 1\pmod n$ for $0 \leq i \leq n-1$. 
\begin{itemize}
    \item Assume $n\equiv 1\pmod 4$, and consider $G' := G^{(n-1)/4}$. For every vertex $i$ in $G'$, $i$ is adjacent to $i \pm j \pmod{n}$ for $1 \leq j \leq \frac{n-1}{4}$. Therefore, we have a $\frac{n-1}{2}$-regular graph. Now we show that this graph is indeed distance critical. Choose some $i \in V(G)$. This vertex has the determining pair $\{i+\frac{n-1}{4}, i-\frac{n-1}{4}\}$. Indeed, $i+\frac{n-1}{4}$ is not adjacent to $i-\frac{n-1}{4}$ and by definition, $i+\frac{n-1}{4}$ is adjacent to $i+\frac{n-1}{4}\pm j$ while $i-\frac{n-1}{4}$ is adjacent to $i-\frac{n-1}{4}\pm j$ for $1 \leq j \leq \frac{n-1}{4}$. Therefore, $i$ is the only common neighbor between $i+\frac{n-1}{4}$ and $i-\frac{n-1}{4}$, and the graph is distance critical as claimed.
    \item Assume $n \equiv 2\pmod 4$, and consider $G^{ (n-2)/4 }$. This is a $\frac{n-2}{2}$-regular graph where vertex $i$ has a determining pair $\{i+\frac{n-2}{4}, i-\frac{n-2}{4}\}$. Further, this is the maximum possible degree since the first paragraph of the proof restricted the degree to at most $\frac{n-1}{2}$. 
    \item Assume $n \equiv 3\pmod 4$, and consider $G^{(n-3)/4}$. This is a $\frac{n-3}{2}$-regular graph where vertex $i$ has the determining pair $\{i+\frac{n-3}{4}, i-\frac{n-3}{4}\}$. Further, this is the maximum possible degree. Notice that $\frac{n-1}{2}$ is odd; therefore, an even number of vertices is required to form a $\frac{n-1}{2}$-regular graph. However, $n \equiv 3\pmod 4$ implies that $n$ is odd.  
    \item Assume $4$ divides $n$, and consider $G^{ (n-4)/4}$. This is a $\frac{n-4}{2}$-regular graph where vertex $i$ has the determining pair $\{i+\frac{n-4}{4}, i-\frac{n-4}{4}\}$. We can; however, make this a $\frac{n-2}{2}$-regular graph by adding an edge between vertices $i$ and $i+\frac{n}{2}\pmod n$. This does not change the fact that vertex $i$ has determining pair $\{i+\frac{n-4}{4}, i-\frac{n-4}{4}\}$ because those two vertices are still nonadjacent and their only common neighbor is still only vertex $i$. Therefore, we have constructed a $\frac{n-2}{2}$-regular distance critical graph. This is the maximum possible degree since the degree is at most $\frac{n-1}{2}$. 
\end{itemize}   
\end{proof}

Next, we examine the properties of edge-maximal distance critical graphs, i.e., distance critical $G$ so that, for all $e \not \in E(G)$, $G+e$ is not distance critical.

\begin{cor}
    Every edge-maximal distance critical graph is connected. 
\end{cor}

\begin{proof}
    This follows immediately from Lemma \ref{lem:diameter}.
\end{proof}

See Table \ref{tab:maxDCgraphs} for the number of edge-maximal distance critical graphs with $n$ small. 

\begin{table}[htbp]
    \centering
    \begin{tabular}{ |c||c|c|c|c|c|c|c|c|c|c|c| } 
			\hline
 $n$ & 1 & 2 & 3 & 4 & 5 & 6 & 7 & 8 & 9 & 10 & 11 \\ 
 \hline
 $\textrm{A}371674(n)$ & 0 &0 & 0 & 0 &1 &1 &2 &4 &14 &82 &557 \\
			\hline
		\end{tabular}
	\caption{The number $\textrm{A}371674(n)$ of edge-maximal distance critical graphs on $n$ vertices (\cite[A371674]{oeis}).}
		
		\label{tab:maxDCgraphs}
	\end{table}

\begin{cor}\label{cor:intersectS}
Let $S$ be the set of vertices which are involved in some determining pair in an edge-maximal distance critical graph, $G$. Then every non-edge of $G$ intersects the set $S$. 
\end{cor}

\begin{proof}
Let $xy \not \in E(G)$. Since $G$ is edge-maximal distance critical, there exists a $z \in V(G)$ such that $z$ does not have a determining pair in $G+xy$. Lemma \ref{lemma:dpstar} ensures that every $z$-determining pair contains either $x$ or $y$; therefore, $xy$ intersects $S$ as desired.
\end{proof}

\begin{cor}
If $G$ is an edge-maximal distance critical graph and $S$ is the set of vertices involved in some determining pair, then the set of vertices $T=V(G)-S$ induces a clique. 
\end{cor}

\begin{proof}
Assume not, i.e., instead there are two vertices $x$ and $y$ of $T$ which are non-adjacent in $G$. We assumed $G$ is edge-maximal, and $xy \notin E(G)$ and does not intersect $S$, which contradicts Corollary \ref{cor:intersectS}. 
\end{proof}

In order to understand better the density of distance critical graphs, we investigate the maximum size of their cliques. Since a determining pair consists of nonadjacent vertices, the determining pair for vertices within the clique have two options: $(1)$ the determining pair consists of one vertex within the clique and one vertex outside the clique, and $(2)$ the determining pair consists of vertices outside the clique. This leads to the following results.  

\begin{definition}
Let $\Gamma_m$ be a graph which has a vertex set divided into $3$ sets: set $A$ has $m(m-1)/2$ vertices labeled as $a_{ij}$ for $0 \leq i < j < m$, set $B$ has $m$ vertices labeled as $b_j$ for $0 \leq j < m$, and set $C$ has $2m$ vertices labeled as $c_j$ for $0 \leq j < 2m$ (understood modulo $2m$). The edge set $E(\Gamma_m) = K \cup X \cup Y \cup Z$, where $K$ consists of all pairs of vertices from $A$; $X$ consists of all pairs $a_{ij} b_i$ and $a_{ij} b_j$, where $0 \leq i \neq j < m$; $Y$ consists of all edges of the form $b_j c_j$ and $b_j c_{j+m}$, where $0 \leq j < m$; and $Z$ contains every edge of the form $c_j c_{j+1}$ for $0 \leq j < 2m$.   
\end{definition} 

The graph $\Gamma_5$ is shown in Figure \ref{fig:G5}. The white vertices lie in set $A$, the gray vertices are in set $B$, and the black vertices are in set $C$.  

		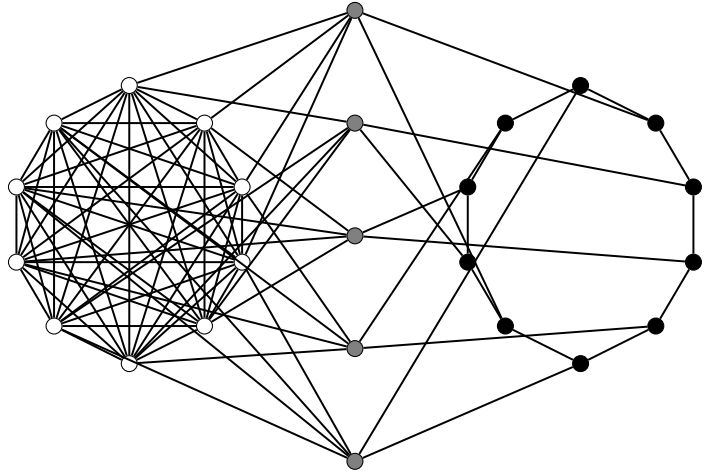
\begin{figure}[htbp]
		\centering
		\begin{tabular}{c}
			\begin{tikzpicture}
			[scale=0.5,inner sep=0.75mm, 
			vertex/.style={circle,draw}, 
			thickedge/.style={line width=0.75pt}] 
			
    \node[vertex] (0) at (0,2) [fill=white] {};
    \node[vertex] (1) at (2,1) [fill=white] {};
    \node[vertex] (2) at (3,-.7) [fill=white] {};
    \node[vertex] (3) at (3,-2.7)[fill=white] {};
    \node[vertex] (4) at (2,-4.4) [fill=white] {};
    \node[vertex] (5) at (0,-5.4) [fill=white] {};
    \node[vertex] (6) at (-2,-4.4) [fill=white] {};
    \node[vertex] (7) at (-3,-2.7)[fill=white] {};
    \node[vertex] (8) at (-3,-.7) [fill=white] {};
    \node[vertex] (9) at (-2,1) [fill=white] {};
    \draw[thickedge] (0)--(1);
    \draw[thickedge] (0)--(2);
    \draw[thickedge] (0)--(3);
    \draw[thickedge] (0)--(4);
    \draw[thickedge] (0)--(5);
    \draw[thickedge] (0)--(6);
    \draw[thickedge] (0)--(7);
    \draw[thickedge] (0)--(8);
    \draw[thickedge] (0)--(9);
    \draw[thickedge] (1)--(2);
    \draw[thickedge] (1)--(3);
    \draw[thickedge] (1)--(4);
    \draw[thickedge] (1)--(5);
    \draw[thickedge] (1)--(6);
    \draw[thickedge] (1)--(7);
    \draw[thickedge] (1)--(8);
    \draw[thickedge] (1)--(9);
    \draw[thickedge] (2)--(3);
    \draw[thickedge] (2)--(4);
    \draw[thickedge] (2)--(5);
    \draw[thickedge] (2)--(6);
    \draw[thickedge] (2)--(7);
    \draw[thickedge] (2)--(8);
    \draw[thickedge] (2)--(9);
    \draw[thickedge] (3)--(4);
    \draw[thickedge] (3)--(5);
    \draw[thickedge] (3)--(6);
    \draw[thickedge] (3)--(7);
    \draw[thickedge] (3)--(8);
    \draw[thickedge] (3)--(9);
    \draw[thickedge] (4)--(5);
    \draw[thickedge] (4)--(6);
    \draw[thickedge] (4)--(7);
    \draw[thickedge] (4)--(8);
    \draw[thickedge] (4)--(9);
    \draw[thickedge] (5)--(6);
    \draw[thickedge] (5)--(7);
    \draw[thickedge] (5)--(8);
    \draw[thickedge] (5)--(9);
    \draw[thickedge] (6)--(7);
    \draw[thickedge] (6)--(8);
    \draw[thickedge] (6)--(9);
    \draw[thickedge] (7)--(8);
    \draw[thickedge] (7)--(9);
    \draw[thickedge] (8)--(9);
    \node[vertex] (10) at (6,4) [fill=gray] {};
    \node[vertex] (11) at (6,1) [fill=gray] {};
    \node[vertex] (12) at (6,-2) [fill=gray] {};
    \node[vertex] (13) at (6,-5) [fill=gray] {};
    \node[vertex] (14) at (6,-8) [fill=gray] {};
    \draw[thickedge] (0)--(10);
    \draw[thickedge] (0)--(11);
    \draw[thickedge] (1)--(10);
    \draw[thickedge] (1)--(12);
    \draw[thickedge] (2)--(10);
    \draw[thickedge] (2)--(13);
    \draw[thickedge] (3)--(10);
    \draw[thickedge] (3)--(14);
    \draw[thickedge] (4)--(11);
    \draw[thickedge] (4)--(12);
    \draw[thickedge] (5)--(11);
    \draw[thickedge] (5)--(13);
    \draw[thickedge] (6)--(11);
    \draw[thickedge] (6)--(14);
    \draw[thickedge] (7)--(12);
    \draw[thickedge] (7)--(13);
    \draw[thickedge] (8)--(12);
    \draw[thickedge] (8)--(14);
    \draw[thickedge] (9)--(13);
    \draw[thickedge] (9)--(14);
    \node[vertex] (15) at (12,2) [fill=black] {};
    \node[vertex] (16) at (14,1)[fill=black] {};
    \node[vertex] (17) at (15,-.7) [fill=black] {};
    \node[vertex] (18) at (15,-2.7)[fill=black] {};
    \node[vertex] (19) at (14,-4.4) [fill=black] {};
    \node[vertex] (20) at (12,-5.4) [fill=black] {};
    \node[vertex] (21) at (10,-4.4) [fill=black] {};
    \node[vertex] (22) at (9,-2.7)[fill=black] {};
    \node[vertex] (23) at (9,-.7) [fill=black] {};
    \node[vertex] (24) at (10,1) [fill=black] {};
    \draw[thickedge] (10)--(16);
    \draw[thickedge] (10)--(21);
    \draw[thickedge] (11)--(17);
    \draw[thickedge] (11)--(22);
    \draw[thickedge] (12)--(18);
    \draw[thickedge] (12)--(23);
    \draw[thickedge] (13)--(19);
    \draw[thickedge] (13)--(24);
    \draw[thickedge] (14)--(15);
    \draw[thickedge] (14)--(20);
    \draw[thickedge] (15)--(16);
    \draw[thickedge] (16)--(17);
    \draw[thickedge] (17)--(18);
    \draw[thickedge] (18)--(19);
    \draw[thickedge] (19)--(20);
    \draw[thickedge] (20)--(21);
    \draw[thickedge] (21)--(22);
    \draw[thickedge] (22)--(23);
    \draw[thickedge] (23)--(24);
    \draw[thickedge] (15)--(24);
		\end{tikzpicture}
		\end{tabular}
	\caption{$\Gamma_5$, with $A$ vertices white, $B$ vertices gray, and $C$ vertices black.}
	\label{fig:G5}
	\end{figure}

\begin{theorem}\label{thm:clique_density}
Among distance-critical graphs $G$ on $n$ vertices, the maximum possible clique number of $G$ is $ n-\Theta\left(\sqrt{n}\right)$.
\end{theorem}

\begin{proof}
Consider the graph $\Gamma_m$ for $m \geq 3$. First, we argue that $\Gamma_m$ is indeed distance critical by noting the determining pairs for each vertex. Consider the vertex $a_{ij}$. This vertex has the determining pair $\{b_i, b_j\}$ because $b_i$ is not adjacent to $b_j$ and $b_i$ is adjacent to $a_{ij}$, $c_i,$ and  $c_{i+m}$ while $b_j$ is adjacent to $a_{ij}$, $c_j$, and $c_{j+m}$. Therefore, $a_{ij}$ is the only common neighbor between $b_i$ and $b_j$. Consider vertex $b_j$. This vertex has determining pair $\{c_j, c_{j+m}\}$ because $c_j$ is not adjacent to $c_{j+m}$ and $c_j$ is adjacent to $b_j$, $c_{j-1}$, and $c_{j+1}$ while $c_{j+m}$ is adjacent to $b_j$, $c_{j+m-1}$, and $c_{j+m+1}$. Therefore, $b_j$ is the only common neighbor between $c_j$ and $c_{j+m}$. Lastly, consider vertex $c_j$. This vertex has determining pair $\{c_{j-1}, c_{j+1}\}$ because $c_{j-1}$ is not adjacent to $c_{j+1}$, and $c_{j-1}$ is adjacent to $c_j$, $c_{j-2}$, and $b_{j-1}$, while $c_{j+1}$ is adjacent to $c_j$, $c_{j+2}$, and $b_{j+1}$. Therefore, $c_j$ is the only common neighbor between $c_{j-1}$ and $c_{j+1}$. 

Next, we establish the clique number. By construction of $\Gamma_m$, a clique of size $m(m-1)/2$ is induced by the vertex set $\{a_{ij}\}_{0\leq i < j < m}$ with a remaining $3m$ vertices of the form $b_j$ or $c_j$. Therefore, $n = \vert V(\Gamma_m)\vert = \binom{m}{2}+3m = m^2/2 + O(m)$ so that $m = \sqrt{2n} (1+o(1))$. From here, we see that $\max_G  \omega(G) \geq n - (3+o(1)) \sqrt{2n} = n-O\left( \sqrt{n}\right)$.

Now we establish a matching upper bound. Consider a distance critical graph $G$ on $n$ vertices. Every vertex $v$ in a max clique $K$ of size $m$ must have a determining pair, say, $\{x_v,y_v\}$. Let $S = \bigcup_{v \in K} \{x_v,y_v\}$.  Note that, for each $v \in K$, $|\{x_v,y_v\} \setminus K| = 1$ or $2$, because if it were zero, then $\{x_v,y_v\} \subset K$ so that $x_v y_v \in E(G)$, contradicting that $\{x_v,y_v\}$ is a determining pair.  Let $D$ be the subset of $V(K)$ with $|\{x_v,y_v\} \setminus K| = 1$ and $E = K \setminus D$.  If $v \in D$, wlog we assume $x_v \in K$ and $y_v \not \in K$.  Note that the $y_v$ are distinct across all $v \in D$, since, if $y_v = y_w$ for some $w \in D$, then $x_v$ and $y_v$ have common neighbors $v$ and $w$, contradicting that they form a determining pair for $v$.  Thus, 
$$
|V(G - K)| \geq |\{y_v : v \in D\}| \geq |D|.
$$
On the other hand, if $v \in E$, then the pair $\{x_v,y_v\}$ is entirely contained in $V(G-K)$.  Since none of these pairs are repeated (or else they could not be determining pairs), the vertex set $\bigcup_{v \in E} \{x_v,y_v\}$ admits at least $|E|$ vertex pairs and therefore at least $\sqrt{2|E|}$ vertices, all of which lie outside $K$.  Therefore, 
$$
|V(G-K)| \geq \max\{|D|,\sqrt{2|E|}\} = \max\{|D|,\sqrt{2(m-|D|)}\} \geq \sqrt{2m+1}-1,
$$
since $0 \leq |D| \leq m$.  Then $G$ contains at least $\sqrt{2m+1}-1$ vertices in addition to the clique $K$, and so $n \geq m + \sqrt{2m} + o(\sqrt{m})$ which implies $m \leq n - \sqrt{(2+o(1)) n}$, and we may conclude that $\omega(G) \leq n - \Omega(\sqrt{n})$.
\end{proof}

The next result builds off of the preceding construction to show that distance criticality is a highly non-local property: every graph is the induced subgraph of some slightly larger distance critical graph.

\begin{theorem} \label{thm:embed_any_G}
Every graph is an induced subgraph of some distance critical graph. 
\end{theorem}

\begin{proof}
Let $G$ be any graph on $n$ vertices. Clearly, we may assume $n \geq 3$.  Modifying the construction of $\Gamma_m$ to obtain a $\Gamma_m'$, let $A$ be $V(G)$, and include $E(G)$ in $E(\Gamma_m')$. An additional $m$ vertices are needed to make up the set $B$ where $m^2 - m - 2n \geq 0$; therefore, add $m = \lceil \frac{1 + \sqrt{1 + 8n}}{2} \rceil$ additional vertices to account for the $B$ set. Label the vertices of $G$ as $a_{ij}$ for the $n$ lexicographically least pairs $\{i,j\}$ with $0 \leq i < j < m$.  Label the added vertices $b_j$ for $0 \leq j < m$. Add edges $a_{ij}b_i$ and $a_{ij}b_j$ for every element $a_{ij}$ of set $A$.  Similarly, add an additional $2m = \lceil 1+\sqrt{1+8n} \rceil$ vertices to make up set $C$ and label these vertices as $c_j$ for $0 \leq j < 2m$ (understood modulo $2m$). Add edges $b_jc_j$ and $b_j c_{j+m}$ for $0 \leq j < m$. Lastly, add edges $c_j c_{j+1}$ for $0 \leq j < 2m$. The argument that every vertex of $\Gamma_m'$ admits a determining pair is nearly identical to the argument for $\Gamma_m$ in the proof of Theorem \ref{thm:clique_density}.  
\end{proof}

We conclude by bounding the edge density of distance critical graphs.

\begin{theorem} \label{thm:edge-density}
Among distance critical graphs $G$ on $n$ vertices, the maximum edge density of $G$ is between $1-O(1/\sqrt{n})$ and $1-\Omega(1/n)$. 
\end{theorem}   

\begin{proof}
    $\Gamma_m$ has $N^2 - \Theta(N^{3/2})$ edges on $N:=m(m+5)/2$ vertices. Further, every vertex of a distance critical graph has a determining pair, ensuring at least $N$ non-edges. Therefore, the maximal edge density is between $1-O(1/\sqrt{n})$ and $1-\Omega(1/n)$. 
\end{proof}

\section{Open Questions} \label{sec:questions}

To the best of our knowledge, the present work introduces distance critical graphs. Therefore, open questions abound. We end by presenting a few particularly interesting ones.  

First, since there is a gap between the upper and lower bounds in Theorem \ref{thm:edge-density}, we ask the following. 

\begin{question}
    What is the maximal edge density of distance critical graphs?
\end{question}

Next, although Theorem \ref{thm:embed_any_G} can be used to obtain a lower bound on the number of distance critical graphs, we were unable to determine if this bound is tight.

\begin{question}
    What fraction of graphs are distance critical?
\end{question}

Finally, the property of distance criticality is easily generalized to finite metric spaces, but the present work did not investigate which of the above results translate easily to that setting, or which ones fail to generalize.  So, we offer the following open-ended question as well.

\begin{question}
    Let $(X,\rho)$ be a finite metric space with the property that the induced metric space on $X-v$ for any $v \in X$ disagrees with $\rho$ on at least one pair of points.  What can be said about such spaces? 
\end{question}

\bibliographystyle{plain}
\bibliography{ref}
\end{document}